\documentclass[a4paper,10pt]{article}
\usepackage[utf8]{inputenc}
\usepackage[utf8]{inputenc}
\usepackage{amsmath,amssymb,amsthm} 
\newcommand{\N}{\mathbb{N}}
\newcommand{\R}{\mathbb{R}}
\renewcommand{\H}{\mathcal{H}} 
\newcommand{\C}{\mathcal{C}} 
\newcommand{\myN}{K}
\newcommand{\myK}{k}
\newcommand{\tr}{\mathrm{tr}}
\newcommand{\B}{\mathcal{B}}
\newcommand{\Bt}{\tilde{\mathcal{B}}}
\newcommand{\Nu}{{\mathcal N}}
\newcommand{\Ra}{{\mathcal R}}
\usepackage{todonotes}

\newtheorem{lemma}{Lemma} 
\newtheorem{theorem}{Theorem} 
\newtheorem{definition}{Definition} 
\newtheorem{corollary}{Corollary} 
\title{On the type of ill-posedness of generalized Hilbert 
matrices and related operators}
\author{Stefan Kindermann\footnotemark[1]}
\date{}

\begin{document}
\footnotetext[1]{Industrial Mathematics Institute, Johannes Kepler University Linz, Alternbergergstraße 69, 4040 Linz, Austria. Email: kindermann@indmath.uni-linz.ac.at}
\maketitle

\begin{abstract}
We consider infinite-dimensional 
generalized Hilbert matrices of the form  $\H_{i,j} = \frac{d_i d_j}{x_i + x_j}$, where 
$d_i$ are nonnegative weights and $x_i$ are pairwise disjoint positive numbers. 
We state sufficient and, for monotonically rearrangeable $x_i$, also necessary conditions
for $d_i$, $x_i$
such that the induced  operator from $\ell^2 \to \ell^2$ and related operators are 
well-defined, bounded, or
compact. Furthermore, we give conditions, when this operator is injective and 
ill-posed. 
\end{abstract}

\section{Introduction}
Throughout this article, let $x_i>0$, $i \in \N$, $x_i \not = x_j$, $i \not = j$, 
be a sequence of disjoint  positive numbers, and let $d_i\geq 0$, $i \in \N$,  be a sequence of 
nonnegative weights. We consider infinite-dimensional 
generalized Hilbert matrices with entries 
\begin{equation}\label{defHil} \H_{i,j} = \frac{d_i d_j}{x_i + x_j} \qquad i,j = 1,\ldots  \end{equation}
Associated with this, we introduce the corresponding operator 
(also denoted  by $\H$) 
on the $\ell^2$-sequence space by the usual infinite-dimensional 
matrix-vector multiplication 
\begin{align*} 
  \H :\ell^2 \to \ell^2:  \qquad 
      (f_i)_i \to \left(\sum_{j=1}^\infty \H_{i,j} f_j\right)_i \, .
\end{align*}
We would like to study the ill-posedness of the corresponding 
operator equation in $\ell^2$. 

The matrix $\H$ is a generalization of the usual well-known Hilbert 
matrix 
\[ \left(\frac{1}{i + j -1}\right)_{i,j=1,\ldots},  
\]
 which is obtained by setting 
\[ d_i = 1, \qquad x_i = i -\frac{1}{2}\, .  
\]
We refer to this  as the {\em standard} setting. 
An instance of the standard Hilbert matrix $x_i = i -\frac{1}{2}$ with the 
weights $d_i = \frac{1}{i}$ has occurred in recent studies on the 
composition of the Hausdorff moment operator and integration
operators~\cite{HofMat22,Gerth22a,KiHo24}.

The Hilbert matrix is a famous object arising, e.g., in the Hausdorff moment problem 
or the Laplace transform. From the perspective of inverse problems, it is 
interesting as it represents a type-1 ill-posed problem, and it has 
been in the focus of recent research, in particular of Bernd Hofmann and 
collaborators~\cite{Gerth21,Gerth22,
Gerth22a,HofMat22}. 

Let us recall that an operator in Hilbert spaces gives rise to an 
ill-posed problem if and only if its range is non-closed 
(by Nashed's definition of ill-posedness~\cite{Nashed86}). 
Nashed has also introduced the distinction between type-I and 
type-II ill-posed problems in \cite{Nashed86}, where type-I problems are those where 
the range of the operator is {\em non-compact} (but non-closed), while 
type-II problems have {\em compact} range (which is necessarily non-closed in 
infinite-dimensions). 
Most of the inverse problems in practice 
are type-II problems, but  inversion of the  Hilbert matrix is 
a type-I problem, which makes the corresponding operator an 
attractive object of research. Although the classical regularization 
theory covers type-I problems, they have some particular uncommon 
features like the lack of a notion of degree of ill-posedness or
the lack  of a regularization by projection~\cite{HofKind10}.

Part of the recent research of 
 Bernd Hofmann deals with these type-I operators, for instance,
the Hausdorff moment problem or the Ces\`{a}ro-Hardy operator and 
the question whether the degree of ill-posedness is destroyed by
combining type-I and type-II operators \cite{HofTau97,HofMat22,Hof06,HofFlei99,Gerth21,Gerth22,
Gerth22a,HeinHof03,Freitag05,DFH24,HofWolf05,HofWolf09,HofKind10}. 

The focus of the current article is to investigate properties  of 
the generalized Hilbert operator and whether it gives rise to a 
type-I or type-II ill-posed problem. This operator, as we show,
is also related to other problems like a generalized moment problem, the sampled 
Laplace transform, and the inversion of the generalized Ces\`{a}ro-Hardy operator.  

We will provide in Section~\ref{sec:2} necessary and sufficient condition for 
the sequences $d_i, x_i$ 
such that  $\H$ is  1.) well-defined, 2.) bounded, and  3.) compact. 
Furthermore, in Section~\ref{sec:3}, we provide sufficient conditions 
that imply injectivity of 
 $\H$ as well as conditions that imply 
ill-posedness of $\H$ in the sense of Nashed,
thus yielding a rather complete  understanding of the type of ill-posedness of $\H$.

\section{Forward properties}\label{sec:2}
We investigate the mapping properties of the generalized Hilbert matrix 
with respect to well-definedness, boundedness, and compactness in $\ell^2$. 
Let us start with some relatives of $\H$.

\subsection{Related operators}
The generalized Hilbert matrix is self-adjoint and arises as 
normal operator of at least two important problems:

Given a sequence of disjoint  numbers $\lambda_n >-\frac{1}{2}$
and nonnegative weights $d_n\geq 0$, 
we define the {\em generalized Hausdorff operator}:
\begin{align*} 
 A : L^2(0,1) &\to \ell^2 \\
  f &\to  \left( d_n  \int_0^1 f(t) t^{\lambda_n} dt\right )_{n} \qquad n = 1,\ldots 
\end{align*}
with adjoint 
\begin{align*} 
 A^* : \ell^2 &\to L^2(0,1)  \\
  (f_n)_{n}&\to \sum_{i=1}^\infty  d_i f_i t^{\lambda_i},
\end{align*}
where the last object is considered a function of $t$. 
It follows that 
\[ \H = A A^* \qquad \text{ with } \quad x_i = \lambda_i + \frac{1}{2} .\]
The standard case is the well-known Hausdorff operator 
that maps a function $f$ to its sequence of 
moments $\left(\int_0^1 f(t) t^{i-1} dt\right)_{i=1}^\infty$.

Related to this problem is the {\em sampled Laplace transform}, where 
functions are mapped to the (weighted) values of the Laplace transform at $x_i$:
The corresponding operator is as follows: 
\begin{align}\label{lap}
\begin{split} 
  L : L^2(0,\infty) &\to \ell^2  \\ 
        f & \to \left(d_n \int_0^\infty e^{-\lambda_n t} f(t) dt\right)_{n} \qquad n = 1,\ldots,  
        \end{split}
\end{align}
and its adjoint 
\begin{align}\label{lapad} 
\begin{split} 
 L^* : \ell^2 &\to L^2(0,\infty) , \\
  (f_n)_{n} &\to \sum_{i=1}^\infty  d_i f_i  e^{-\lambda_i t} \, .
   \end{split}
\end{align}
We have again 
\begin{equation}\label{normallap}
\H = L L^* \qquad \text{ with}  \qquad x_i = \lambda_i.   \end{equation}
Thus, the $x_i$ in the definition of $\H$ have the interpretation 
of  sampling points of the Laplace transform. 
The standard case refers to sampling at the midpoints of the integers. 
Of course, the Hausdorff problem and the sampled Laplace transform are 
related by a simple change of variables.

For the analysis we also need certain summation operators. 
Given the sequences $d_i,x_i$, 
we define the generalized Ces\`{a}ro-Hardy operator by 
\begin{align}\label{ces}
\begin{split}
 \C:  \ell^2 & \to \ell^2 \\ 
 (f_n)_{n=1}^\infty &\to  \left( \frac{d_n}{x_n} \sum_{x_k \leq x_n} d_k f_k  \right)_{n=1}^\infty   
 \end{split}
\end{align} 
as well as        its adjoint 
\begin{align}\label{cesad}
\begin{split}
 \C^*:  \ell^2 & \to \ell^2 \\ 
       (f_n)_{n=1}^\infty &\to   \left( d_n \sum_{  x_k  \geq x_n } \frac{d_k f_k}{x_k} \right)_{n=1}^\infty  \, .
 \end{split}
\end{align}

In the (almost) standard case $d_i = 1, x_i = i$, $i \in \N$,
the Ces\`{a}ro-Hardy operator is the well-known cumulative mean (Ces\`{a}ro mean):
For $f = (f_n)$,
\[ (\C f)_n = \frac{1}{n} \left(f_1 + f_2 \cdots  + f_n\right)  \qquad n = 1,\ldots \]
Boundedness of this operator is the essence of the 
Hardy inequality \cite[Section~9.8]{Polya}.
In the above-mentioned almost standard case, the 
Ces\`{a}ro-Hardy operator   $\C$ represents an 
ill-posed operator that is non-compact (i.e., type-I), 
(this follows from  \cite{Brown65}, where the spectrum was identified but also from our analysis.)
It shares this property  with the continuous 
Ces\`{a}ro-Hardy operator, where sums are replaced by integrals:  
\[ \C_c: L^2(0,1) \to L^2(0,1) \qquad f \to \frac{1}{t} \int_0^t f(s) ds. \]
Its ill-posedness was investigated in \cite{DFH24}, and a ``curious'' composition 
with $A$ was considered in \cite{KiHo24}.

Let us remark the following useful observation: 
For a given generalized Hilbert matrix with $x_i,d_i$, we may rewrite 
the entries as 
\begin{equation}\label{trans} \H_{i,j} = \frac{d_i d_j}{x_i + x_j} = 
 \frac{\frac{d_i}{x_j} \frac{d_j}{x_j} }{\frac{1}{x_i} + \frac{1}{x_j}} \, ,
 \end{equation}
and thus consider the same matrix generated by 
the sequences $x_i^{-1},\frac{d_i}{x_i}$  in place of 
 $x_i,d_i$. The conditions in the main theorem in the next section are completely invariant under 
 this transformation. 
 
An interesting observation is that, 
denoting by $\C_{x,d}$ the Ces\`{a}ro-Hardy operator in \eqref{ces} 
generated by $x_i,d_i$, then 
the operator generated by the transformed sequence satisfies 
$\C_{x^{-1},\frac{d}{x}} = \C^*_{x,d}$.

\subsection{Main result}
We now investigate well-definedness, boundedness, and compactness of 
$\H$. By well-definedness we mean the property that 
$(\H f)_n < \infty$ for any sequence $(f)_n \in \ell^2$ and any index $n \in \N$. 

For later use we state a useful inequality valid for $x,y \geq 0$: 
\begin{equation}\label{useful}
\frac{1}{2} \left(  \frac{1}{x} \chi_{y < x}(x,y) +   
        \frac{1}{y} \chi_{y \geq  x}(x,y) \right) \leq 
\frac{1}{x + y} \leq 
     \frac{1}{x} \chi_{y < x}(x,y) +   
        \frac{1}{y} \chi_{y \geq  x}(x,y),         
\end{equation}        
where $\chi_A$ is the indicator function of the set $A$. This inequality 
follows easily from 
\[ \max\{x,y\} \leq (x+y) \leq 2\max\{x,y\}. \]
For the main theorem we need  the following functions: 
Given the sequences $d_i,x_i$, we define for $x \geq 0$
\begin{align}\label{defNM} 
 N(x) := \sum_{x_j \leq x} {d_j^2},   \qquad 
 M(x) :=  \sum_{x_j \geq  x} \tfrac{d_j^2}{x_j^2}. 
\end{align}
For a $K \in \N$, we define discrete analogues of it:
\begin{align}\label{defNMdis} 
 N^K(x) := \sum_{\substack{ x_j \leq x \\  1 \leq j \leq K }} {d_j^2},   \qquad 
 M^K(x) :=  \sum_{\substack{ x_j \geq  x\\  1 \leq j \leq K}} \tfrac{d_j^2}{x_j^2}. 
\end{align}
It is clear that $N^K(x),$ $M^K(x)$ is monotonically increasing in $K$,  
that $N^K(x),$ $N(x)$ are increasing in $x$, and that 
$M^K(x),M(x)$ are decreasing in $x$, and we have for any~$x$ 
\[ \lim_{K\to\infty} N^K(x) = N(x) \qquad 
 \lim_{K\to\infty} M^K(x) = M(x). 
\]
 Note that 
\begin{alignat}{2} N^K(x) &= 0 \quad \text{ for } x < \min_{1 \leq i \leq K} x_i,& 
 \qquad
 N^K(x) &\leq N^K(\max_{1 \leq i \leq K} x_i),  \label{upperN}  \\ 
 M^K(x) &= 0 \quad \text{ for } x  > \max_{1 \leq i \leq K} x_i, &\qquad 
 M^K(x) &\leq M^K(\min_{1 \leq i\leq K} x_i).  
\end{alignat} 

We note that the definitions of $M$, $N$  are 
invariant under rearrangements of $x_i$. Moreover, such rearrangements 
do not alter the analytic properties of $\H$ (boundedness and compactness) 
since rearrangements just induce an unitary operator on $\ell^2$.

For later use we need the  following definition:
\begin{definition}
A sequence $x_i >0$ with disjoint elements is monotonically rearrangeable (m.r.) if 
there is a rearrangement of indices (i.e., a bijection $\sigma:\N \to \N$) 
such that 
\[ x_{\sigma(1)} < x_{\sigma(2)} < \ldots \]
or 
\[ x_{\sigma(1)} > x_{\sigma(2)} >\ldots \]
\end{definition} 
Obviously, a strictly increasing sequence is (m.r.), 
and a subsequence of a (m.r.) sequences is also (m.r.).

We are now in the position to state the main result.
The analytic properties of $\H$ can now be characterized
as follows:
\begin{theorem}\label{mainth}
For $d_i,x_i$ as stated above, define for $L\in \N$
  \begin{equation}\label{btdef} 
       \Bt(L) :=      \sup_{L\leq K \in \N}\left\{N(\max_{L\leq i\leq K} x_i)
  M(\min_{L\leq K \leq  i} x_i) \right\}  + \sup_{L\leq k \in \N} N(x_k) M(x_k) .
  \end{equation}
Then the operator $\H:\ell^2 \to \ell^2$ is 
\begin{enumerate}
 \item well-defined if and only if 
 \begin{equation}\label{well} \max\{N(x_n), M(x_n)\} < \infty \qquad \forall n \in \N; 
 \end{equation}
 \item well-defined and bounded  if  
  \begin{equation}\label{second} 
  \B(1) < \infty; 
  \end{equation}
  \item compact if it is   bounded and
  \begin{equation}\label{third}  
    \lim_{\substack{L\to \infty}} \Bt(L)  = 0. 
  \end{equation}
\end{enumerate}   
   In case that $x_i$ is monotonically rearrangeable, 
   we have that $\H$ is 
   \begin{enumerate}
     \item[2a.]  well-defined and bounded if and only if   
   \begin{equation} \label{secondb} \sup_{k\in \N} N(x_k) M(x_k) < \infty; 
   \end{equation}
     \item[3a.] compact if and only if it is bounded and 
       \begin{equation}\label{thirdb} \lim_{\substack{L\to \infty}} 
        \sup_{k\geq L } N(x_k) M(x_k) \to 0 . 
  \end{equation}
  \end{enumerate} 
\end{theorem}

Before proving this theorem, let us briefly illustrate the results. In the standard case, for the 
usual Hilbert matrix $d_i = 1,$ $x_i = i -\frac{1}{2}$, and $x_i$ is rearrangeable since it is 
monotonically increasing. 
We have 
that $N(x) \sim x$ and $M(x) \sim \int_x^\infty \frac{1}{s^2} ds \sim x^{-1}$
and thus $N(x) M(x) \sim 1$. 
Thus, we get the well-known result that the Hilbert matrix in the standard setting 
is well-defined and bounded but not compact by \eqref{thirdb}.

Using the setting $x_i = i-\frac{1}{2}$, $d_i = \frac{1}{i}$ as in \cite{KiHo24}, 
leads to $N(x) \sim \int_1^x \frac{1}{s^2} ds \sim O(1)$, while 
$M(x) \sim \int_x^\infty \frac{1}{s^4} ds \sim \frac{1}{x^3}$. 
In this case, $N(x) M(x) \sim \frac{1}{x^3}$ for $x>1$, 
and for $x = x_L =L- \frac{1}{2}$, this expression tends to $0$ as $L\to \infty$,
and \eqref{thirdb} holds. 
Hence, this Hilbert matrix  operator is compact,  as   also  shown in \cite{KiHo24}.

Furthermore,  we  should state that these results are 
probably partially  known  (or at least expectable) to experts, and we do not claim full originality of 
central parts of the proofs. 
In particular, the method of applying the Ces\`{a}ro-Hardy operator 
for proving boundedness is the celebrated idea of Hardy~\cite{Polya}. The 
condition with $N(x) M(x)$ are the famous Muckenhoupt condition for 
the boundedness of the weighted Hardy operator, and our boundedness result is (probably) 
partially contained in  \cite{Muck}, but the translation of the 
proofs in \cite{Muck} to the sequence space was not obvious to the author. 
Note that such was done in \cite{Miclo}, 
and our proof of item 2.)  follows that in  parts.
However, since we consider infinite sequences $x_i$ that are not necessarily 
ordered, it was not completely clear if the steps in \cite{Miclo} apply to this case, 
so   we provide a selfcontained elaboration.

Moreover, let us mention that a main source of inspiration for this 
work is that of Ushakova \cite{Usha13,StepUsha04,Usha11,Usha13} and also the references therein. 
In~\cite{Usha13}  similar
boundedness and compactness condition were
derived for the Laplace transform in weighted $L^p$-spaces,
which can be considered the 
continuous analogue of our $\ell^2$-setting. 
We note that  the results in our article most probably can be extended 
to $\ell^p$-spaces with $p\in (1,\infty)$, analogous as in  
\cite{Usha13} for the $L^p$-case, 
 but we do not consider this here.

\subsection{Proof of well-posedness} 
We are concerned with the well-posedness of $\H$.
Take a sequence $(a_n)_n \in \ell^2$. We use the 
Cauchy-Schwarz estimate 
\begin{align*}  
 |(\H a)_n|^2 = |\sum_{j=1}^\infty \frac{d_n d_j}{x_n + x_j} a_j |^2 \leq 
  d_n^2 \sum_{j=1}^\infty \frac{d_j^2}{(x_n + x_j)^2} \|a\|_{\ell^2}^2.   
\end{align*}
Since we may take a sequence where the Cauchy-Schwarz inequality is sharp, 
$a_k \sim \frac{d_k}{x_n + x_k}$, we have that $\H$ is well-defined if and 
only if 
\[ \sum_{j=1}^\infty \frac{d_j^2}{(x_n + x_j)^2} < \infty \qquad \forall n \in \N. \]
By \eqref{useful}, we have 
\begin{align}\label{seconduseful}
   \frac{1}{4} \left(\frac{1}{x_n^2} \sum_{x_j < x_n}  d_j^2 + 
   \sum_{x_j \geq  x_n}  \frac{d_j^2}{x_j^2}  \right)    \leq 
   \sum_j \frac{d_j^2}{(x_n + x_j)^2}  \leq 
  \frac{1}{x_n^2} \sum_{x_j < x_n}  d_j^2 + 
   \sum_{x_j \geq  x_n}  \frac{d_j^2}{x_j^2}  .
\end{align}
Hence, necessary and sufficient for convergence of the sum is 
\[  \frac{1}{x_n^2} \sum_{x_j < x_n}  d_j^2 < \infty \qquad \mbox{ and }  \qquad  
   \sum_{x_j \geq  x_n}  \frac{d_j^2}{x_j^2}   < \infty \qquad \forall n,  
\]
and an equivalent condition is obtained by replacing the first sum over 
$x_j \leq x_n$ as this only involve the additional term $d_n^2$. 
Thus, in terms of the function \eqref{defNM}, this means that 
 $N(x_n)$ and $M(x_n)$ are finite at any point, 
which is exactly the stated condition \eqref{well}.

\subsection{Proof of boundedness}
For the proof of boundedness in item 2.) (condition \eqref{second}) 
and  in item 2a.) (condition \eqref{secondb}) 
we require the next lemma, cf.~\cite[p.~241]{Polya},
which is  valid for finite sequences:

\begin{lemma}
For a finite sequence 
$(x_i)_{i=1}^K$, corresponding positive weights $w_k >0$, and an arbitrary $x\geq0$,
  we have the inequality 
\begin{align}  
\sum_{ x_n \leq x} \frac{w_n}{\sqrt{\sum_{x_j \leq x_n}   w_j }} 
 &\leq 2 \sqrt{\sum_{x_j \leq x}   w_j }. \label{one} 
\end{align} 
\end{lemma}
\begin{proof}
We note that the inequality only involves finitely many $x_i \leq x$, $1 \leq i \leq K$,
and the left- and 
right-hand side are invariant under rearrangements of these. Thus, 
we may assume now that these are monotonically increasing 
\[ x_1 < x_2 < \ldots < x_L \leq x  < \min_{L < i \leq  K} x_i . \] 
Let  $b\geq a \geq 0$, $b>0$. Then, 
\begin{align*}
\sqrt{b} - \sqrt{a} = 
\frac{b-a}{\sqrt{b} + \sqrt{a}} \geq 
  \frac{b - a}{2 \sqrt{b}} .
\end{align*}
For $x_n$, $2 \leq n \leq L$, taking 
$b = \sum_{x_j\leq x_n} w_j$, $a = \sum_{x_j\leq x_{n-1}} w_j$ yields 
\[ \sqrt{\sum_{x_j\leq x_n} w_j } - 
 \sqrt{\sum_{x_j\leq x_{n-1}} w_j}  \geq 
 \frac{ w_n}{2  \sqrt{\sum_{x_j\leq x_n} w_j }}. 
\]
Summing this identity over $n=2,\ldots,L$ yields a telescope sum, hence 
\[  \sqrt{\sum_{x_j\leq x_L} w_j } - 
\sqrt{\sum_{x_j\leq x_1} w_j } \geq 
\frac{1}{2}  \sum_{n=2}^{L} \frac{ w_n}{ \sqrt{\sum_{x_j\leq x_n} w_j }}.
\]
Adding the following  term corresponding to that in the sum  with $n = 1$,  
\[ \frac{1}{2}  \frac{ w_1}{ \sqrt{\sum_{x_j\leq x_1} w_j }} = 
\frac{\sqrt{w_1}}{2} =  \frac{1}{2} \sqrt{\sum_{x_j\leq x_1} w_j }\]
yields 
\begin{align*}
 \sqrt{\sum_{x_j\leq x} w_j } 
=   \sqrt{\sum_{x_j\leq x_L} w_j }
 \geq 
 \sqrt{\sum_{x_j\leq x_L} w_j } - \sqrt{\sum_{x_j\leq x_1} w_j } +
\frac{1}{2} \sqrt{\sum_{x_j\leq x_1} w_j } \\
 \geq 
\frac{1}{2}  \sum_{n=1}^{L} \frac{ w_n}{ \sqrt{\sum_{x_j\leq x_n} w_j }}
= \frac{1}{2}  \sum_{x_n \leq x} \frac{ w_n}{ \sqrt{\sum_{x_j\leq x_n} w_j }},
\end{align*}
which completes the proof.
\end{proof}

Applying this lemma to $N^K(x_n)$ (with $w_j = d_j^2$)  and 
$M^K(x_n)$ (with $w_j = \frac{d_j^2}{x_j^2}$)
we have for arbitrary $x\geq 0$,
\begin{align} \sum_{\substack{x_k \leq x\\1 \leq k \leq K}} 
\frac{d_k}{N^K(x_k)^\frac{1}{2}}   
\leq 2 N^K(x)^\frac{1}{2} , \label{MNest} \\ 
\sum_{\substack{x_n \geq x \\1 \leq n \leq K}} \frac{d_n^2}{x_n^2}M^K(x_n)^\frac{1}{2}   
\leq 2  M^K(x)^\frac{1}{2}.  \label{MMest}
\end{align}
The last inequality is obtained by transforming the finite sequence $(x_i)_{i=1}^K$ 
to $(x_i^{-1})_{i=1}^K$ and applying the lemma then.


We continue with the proof of item 2) by 
showing that \eqref{second} implies boundedness of $\H$.

Let $f,g$ be sequences in $\ell^2$. We denote by $|f|$, $|g|$ the respective sequences
with absolute value applied 
componentwise. 
By \eqref{useful} and 
the definition of \eqref{ces}, \eqref{cesad}, 
it follows that
\begin{equation}\label{xyz} 
 \begin{split}
 (\H f,g) &\leq 
 \sum_{n,m} \frac{|f_n| |g_m| d_n d_m}{x_n + x_m}  
\leq 
  \sum_{x_m \leq x_n} \frac{|f_n| |g_m| d_n d_m}{x_n} +
   \sum_{x_m \geq  x_n} \frac{|f_n| |g_m| d_n d_m}{x_m} \\
  &=  (|g|, \C^* |f|)_{\ell^2} + 
  (\C^* |g|,|f|)_{\ell^2}   \leq 
      \big\| \, \C |g| \, \big\|_{\ell^2}\|f\|_{\ell^2} +
    \big\| \,\C |f| \,   \big\|_{\ell^2}\|g\|_{\ell^2}. 
\end{split}
\end{equation}
Let us estimate the norm of $\C f$ for $f$ being a finite sequence,
$f_i = 0$ for $i >K$.
\begin{align} \|\C f\|^2 &= 
  \sum_{n=1}^\infty \frac{d_n^2}{x_n^2} 
 \left(\sum_{\substack{x_k \leq x_n\\ k\leq K}} d_k f_k \right)^2 \label{first} \\
& =   \sum_{n=1}^\infty \frac{d_n^2}{x_n^2} 
 \left(\sum_{\substack{x_k \leq x_n\\ k\leq K}} d_k f_k \frac{N^K(x_k)^\frac{1}{4}}{d_k}  
 \frac{d_k} {N^K(x_k)^\frac{1}{4}}
 \right)^2    \notag \\
&\leq_{\text{C.S. ineq.}}    \sum_{n=1}^\infty \frac{d_n^2}{x_n^2} 
 \left(\sum_{\substack{x_k \leq x_n\\ k\leq K}} d_k^2 f_k^2 \frac{N^K(x_k)^\frac{1}{2}}{d_k^2}  \right)
 \left( \sum_{\substack{x_j \leq x_n\\ j\leq K}}  \frac{d_j^2} {N^K(x_j)^\frac{1}{2}}
 \right) \notag   \\
& \leq_{\text{ by } \eqref{MNest}} 2 \sum_{n=1}^\infty \frac{d_n^2}{x_n^2} 
 \left(\sum_{\substack{x_k \leq x_n\\ k\leq K}}   f_k^2 {N^K(x_k)^\frac{1}{2}}  \right)
 N^K(x_n)^\frac{1}{2} \\
& = 2 \sum_{k=1}^\infty \sum_{n=1}^\infty\frac{d_n^2}{x_n^2} 
\chi_{\left\{\substack{x_k \leq x_n\\k\leq K}\right\}} 
f_k^2 {N^K(x_k)^\frac{1}{2}}   N^K(x_n)^\frac{1}{2} \notag  \\
&\leq  2 \sum_{k=1}^K 
f_k^2 {N^K(x_k)^\frac{1}{2}}
\left(\sum_{x_k \leq x_n}^\infty\frac{d_n^2}{x_n^2} 
   N^K(x_n)^\frac{1}{2} \right).   \label{herin}
\end{align} 
Recall the definition of $\Bt(1)$ in \eqref{btdef} and its use in \eqref{second}.
We have that 
\[ M^K(x_n) N^K(x_n) \leq 
 M(x_n) N(x_n) \leq \Bt(1). 
\]
Now we bound the term in brackets in \eqref{herin}.
Note that for $x_n$ with $n \leq K$, 
we have that $M^K(x_n)$ is nonzero. 
\begin{align}
&\sum_{x_k \leq x_n}^\infty\frac{d_n^2}{x_n^2} 
   N^K(x_n)^\frac{1}{2} 
   = 
\sum_{\left\{ n:\substack{x_k \leq x_n \\ n \leq K }\right\} } \frac{d_n^2}{x_n^2}
  N^K(x_n)^\frac{1}{2}  + 
  \sum_{\left\{n: \substack{x_k \leq x_n \\ n > K } \right\}}^\infty\frac{d_n^2}{x_n^2}
  N^K(x_n)^\frac{1}{2} \label{herinin}
  \\
& \leq_{\eqref{upperN}}  \sum_{\left\{ n:\substack{x_k \leq x_n \\ n \leq K }\right\} }\frac{d_n^2}{x_n^2}
  \frac{M^K(x_n)^\frac{1}{2} N^K(x_n)^\frac{1}{2}}{M^K(x_n)^\frac{1}{2}}  + 
  N^K(\max_{i\leq K} x_i)^\frac{1}{2} \sum_{\left\{n: \substack{x_k \leq x_n \\ n > K } \right\}}^\infty\frac{d_n^2}{x_n^2} \notag\\
& \leq  \Bt(1)^\frac{1}{2}  
\sum_{\left\{ n:\substack{x_k \leq x_n \\ n \leq K }\right\} }
\frac{\frac{d_n^2}{x_n^2}}{M^K(x_n)^\frac{1}{2}} + 
  N^K(\max_{i\leq K} x_i)^\frac{1}{2} \sum_{n > K }^\infty\frac{d_n^2}{x_n^2}    
  \\
  & \leq_{\eqref{MMest}} 
  2 \Bt(1)^\frac{1}{2}  M^K(x_k)^\frac{1}{2} +   N(\max_{i\leq K} x_i)^\frac{1}{2} M(\min_{n >K} x_n).
\end{align}
In total, 
\begin{align}\label{aaa}
  \|\C f\|^2  & \leq 2 \sum_{k=1}^K  f_k^2
  \left( 
   2  \Bt(1)^\frac{1}{2}  M^K(x_k)^\frac{1}{2} 
   {N^K(x_k)^\frac{1}{2}}+   N(\max_{i\leq K} x_i)^\frac{1}{2} M(\min_{n >K} x_n) \right), 
\end{align}
and since $ M^K(x_k)^\frac{1}{2}   N^K(x_k)^\frac{1}{2} \leq 
 M(x_k)^\frac{1}{2} N(x_k)^\frac{1}{2} \leq \Bt(1)$,
 and the definition of $\Bt(1)$, 
this shows that $\C f$ is uniformly bounded on finitely supported $f$. 
Let us denote by 
\mbox{$P_M:\ell^2 \to \ell^2$} 
 the projector onto the subspace where only the first $M$ components of a sequence
are nonzero, i.e., 
\begin{equation}\label{prj} (P_M f)_{n=1}^\infty = (f_1,\dots, f_M,0,0,\ldots).   \end{equation} 
Then the just proven result means that 
$ \| C P_M \|$ is uniformly bounded (by $3\Bt(1)^\frac{1}{2}$).
Moreover, since $ P_M f = f$ for sufficiently large $M$ on 
such finitely supported $f$, it follows that 
$\lim_{M\to \infty} C P_M f = C f$ . By the Banach Steinhaus 
principle this means that $\|C\|$ is bounded 
with the bounds proportional to $\Bt(1)$.

It is obvious that for $x_i$ increasing, we have 
$\max_{i\leq K} x_i = x_K$ and $\min_{i\geq K} x_i = x_K$, 
such that the condition \eqref{secondb} is equivalent to \eqref{second} in case 
that $x_i$ is monotonically increasingly rearrangeable.
In case $x_i$ being decreasing, we transform the problem as stated above 
in \eqref{trans}  using  the  increasing sequence $x_i^{-1}$. 
Thus, since rearrangments do not change the norm of $\H$, 
we get boundedness under condition~\eqref{secondb}.

Now consider the necessity of \eqref{secondb} in the increasing 
case. Assume that 
$x_i$ is increasing and that $\H$ is bounded and well-defined,
and we want to deduce that~\eqref{secondb} holds.

Take, for some $x_0 \in \{x_i: \, i \in \N\}$,
the sequences $f,g$ defined by 
\[ f_j := \chi_{x_j \geq x_0} \frac{d_j}{x_j},  \qquad 
g_j:= \chi_{x_j \leq x_0} d_j.  \]
Then, 
\[ \|f\|_{\ell^2}^2 = M(x_0), \qquad \|g\|_{\ell^2}^2 = N(x_0). \]
By well-definedness both norms are finite.  
The sequences contain nonnegative numbers, and by \eqref{useful}, we obtain that 
\[ \frac{\chi_{x_m < x_n} d_m d_n f_m 
 g_n }{x_n} + 
  \frac{\chi_{x_n \leq x_m} d_m d_n f_m 
 g_n }{x_m}  \leq 2 \frac{d_m d_n f_m g_n}{x_m + x_n}. \]
A summation over $m$ and $n$ yields 
\[ 2 (\H f,g) \geq 
\sum_{m,n} \frac{d_m^2 d_n^2}{x_m x_n} \chi_{x_m < x_n} \chi_{x_m \geq x_0}
 \chi_{x_n \leq x_0}+ 
\sum_{m,n} \frac{d_m^2 d_n^2}{x_m^2} \chi_{x_m \geq  x_n} \chi_{x_m \geq x_0}
 \chi_{x_n \leq x_0}.
 \]
The first sum on the right-hand side 
is empty since the inequalities for $x_m,x_n,x_0$ are contradictory. 
In the second sum, the condition $x_m \geq  x_n$ is implied by the other ones, thus 
\begin{align*} 2 (\H f,g) &\geq
2 \sum_{m,n} \frac{d_m^2}{x_m^2} d_n^2 \chi_{x_m \geq x_0} \chi_{x_n \leq x_0}
= \left( \sum_{m\geq x_0} \frac{d_m^2}{x_m^2} \right) 
\left(\sum_{n\leq x_0} d_n^2 \right) \\
&= M(x_0) N(x_0). 
\end{align*}
By assumption, $\H$ is bounded, thus 
\[M(x_0) N(x_0) \leq  2 (\H f,g) \leq 2 \|\H\| \|f\|\|g\| =  2 \|\H\| 
M(x_0)^\frac{1}{2} N(x_0)^\frac{1}{2}. \]
Combining the inequalities and taking $x_0 = x_n$ arbitrary yields that 
\begin{equation}\label{bound}  M(x_n) N(x_n) \leq (2 \|\H\|)^2, \end{equation}  
 and hence, \eqref{second} holds.

\subsection{Proof of compactness}
Let us continue with the proof of compactness under 
condition \eqref{third}. 

Assume that $\H$ is bounded and \eqref{third} holds true. 
 Let $P_\myN$ be the 
$\ell^2$-projector onto the first $\myN$ coefficients as in \eqref{prj}
and consider the Laplace transform \eqref{lap}. 
Since $\H = L L^*$, $L$ is bounded and so is 
the operator $P_\myN L$. The range of this operator involves only the 
first $\myN$ components of a sequence and is hence finite-dimensional, 
and therefore $P_\myN L$ is compact. 
Consider 
\[ Q_\myN:=  P_\myN L- L \quad \text{and} \quad \H_{\myN}:= Q_\myN Q_\myN^* .\]
The operator  $Q_\myN$ and hence $\H_{\myN}$ only involves 
the subsequences $(x_i)_{i >K}$, $(d_i)_{i>K}$, and a 
easy calculation shows that, as an infinite matrix, 
\[ \H_{\myN} = \begin{pmatrix} 0 & 0 \\ 0 & (\H_{ i,j})_{i,j>K} \end{pmatrix} \]
i.e., it equals $\H$ but with $i\leq K$ or $j \leq K$ set to $0$. 
The norm of $\H_{\myN}$ equals the norm of the submatrix 
$(\H_{ i,j})_{i,j>K}$, which is a generalized Hilbert matrix such that 
the norm bound of the previous section applies,
\begin{align}\label{thisesti} \|\H_{\myN}\| \leq 
\Bt(K). 
\end{align} 
By condition \eqref{third}, this expression tend to $0$, hence 
     $\lim_{\myN \to \infty} \|P_\myN L- L\| = 0$, and since 
     the uniform limit of a compact operator is compact, $L$ and 
     hence $\H = L L^*$ must be compact, too. 

For the converse, assume that $x_n$ is monotonically (with loss of generality) 
increasing, and let $\H$ be bounded  and compact. 
The projection operator $P_\myN$ converges pointwisely to the identity 
as $\myN \to \infty$. 
It is  a well-known result that 
 if $L$ is a compact operator 
 then $\|(P_\myN -I)L\| \to 0$ 
(see, e.g.,~\cite[Cor. 10.5]{kress}, \cite{engl}). 
If $\H$ is compact, so is $L$. Thus, 
$ \|(P_\myN -I)L\| \to 0$, as $\myN\to \infty$,
and by the converse result \eqref{bound},
\[ \|(P_\myN -I)L\|^2 = \|Q_{\myN} Q_{\myN}^*\| = \|\H_{\myN}\| 
\geq C  \sup_{i\geq K } N(x_i) M(x_i). \]
Since the left-hand side tends to $0$, this yields \eqref{thirdb}.

\qed

The results of Theorem~\ref{mainth} immediately transfer to  the other related operators. 
\begin{theorem}
The statements on boundedness and compactness, items 2.) 3) 2a), 3a.) remain 
valid when $\H$ is replaced by $A$, $A^*$, $A^*A$, $L$, $L^*$, $L^*L$,
$\C$, $\C^*$, $\C \C^*$, $\C^*\C$. 
\end{theorem}
\begin{proof}
The first part  involving $A$ and $L$ and (combinations of) its adjoints  
is obvious since it is a fact that for any operator $A$ in Hilbert spaces, 
$AA^*$ is bounded or compact if and only if 
either of $A$, $A^*$, or  $A^*A$ is (cf., e.g., \cite{KiHo24}).

For the second part, involving the Ces\`{a}ro-Hardy operators, we observe that 
we have shown in \eqref{aaa} that $\|\C\| \leq C \Bt(1)$ 
which gives 2.). 
Concerning 2a.), in the case when $\C$ is bounded, 
we mention that
by \eqref{xyz}, $\|\C\| < \infty$  implies $\|\H\| <\infty$, which implies 
\eqref{secondb}. This proves statement 2a). 

Concerning 3.) and 3a.). 
Assume that the compactness conditions 
\eqref{third} or \eqref{thirdb} or  hold. Then $\H$ is compact. 
let $f_K:= (I -P_K) f$   and observe that $|f_K| = (I -P_K) |f|$. 
Then, as $\C$ involves summation with positive weights, 
\begin{align*} |(\C (I -P) f,g)|  &= |(\C f_K,g)| \leq (\C |f_K|, |g|) \leq 
  (\C |f_K|, |g|) + (\C |g|, |f_K|) \\
  &\leq 
  (\H |f_K|,|g|) = 
    (\H (I -P_K) |f|,|g|)  \leq   \|\H (I -P_K)\|  \|f\| \|g\|.
\end{align*}
By compactness of $\H$, $\|\H(I -P_K)\|$ converges  to $0$ as $K\to \infty$, 
which implies that $\|\C(I -P_K)\|$ does, hence 
$\C$ is compact. 

It remains to show that compactness 
of $\C$ implies \eqref{secondb}. If $\C$ is compact, then so is $\C^*$. 
Thus, by \eqref{xyz}, 
\begin{align*} (\H f_K,g) &\leq  
((\C + \C^* )|f_K|,|g|) \leq 
 \|\C |f_K| \| \|g\| +  \|\C^* |f_K| \| \|g\|  \\
 &  = \|\C (I -P_K) |f| \| \|g\| +  \|\C^* (I -P_K)|f| \| \|g\|  \\
 & \leq 
 \|f\| \|g\| \left(\|\C (I - P_K)\| + \|\C^* (I - P_K)\|\right). 
\end{align*}
By compactness, the right-hand side tends to $0$ for bounded $f,g$, thus the norm 
$\|H (I -P_K)\|$ does, hence $\H$ is compact, hence \eqref{secondb} must 
be satisfied. 
\end{proof}


\section{Inverse properties}\label{sec:3}
We now study properties of $\H$ with respect to 
inverse problems; specifically we state conditions 
when $\H$ is injective and when $\H$ has a non-closed range. 
In this section we denote by $\Nu$ and $\Ra$ the nullspace and 
the range, respectively. 

\subsection{Injectivity} 
The first investigation  concerns the question if $\H$ has a nullspace. 
\begin{lemma}\label{lemmah1}
Let $f  = (f_n) \in \ell^2$ and let $x_i$, $d_i$ be as above. 
For any $s >0$ and any $x >0$,
 we have the estimate  
\[ \sum_{i=1}^\infty |f_i d_i| e^{- s x_i} \leq 
\|f\|_{\ell^2} (N(x)^\frac{1}{2} 
\sup_{x_i \leq x} \left( e^{- s x_i} \right)+ 
\|f\|_{\ell^2}  M(x)^\frac{1}{2} \sup_{x_i \geq x} \left( x_i e^{- s x_i} \right) \,.
\]
\end{lemma}
\begin{proof}
For $N$ fixed, we estimate by the Cauchy-Schwarz inequality 
\begin{align*} \sum_{i=1}^N &|f_i d_i| e^{-s x_i}= 
\sum_{x_i \leq x, i\leq N } |f_i| d_i  e^{-s x_i}+ 
\sum_{x_i \geq x, i\leq N } |f_i| \frac{d_i}{x_i} x_i e^{-s x_i} \\ 
 & \qquad  \leq 
\left(\sum_{x_i \leq x, i\leq N } f_i^2 \right)^\frac{1}{2} 
\left(\sum_{x_i \leq x, i\leq N } d_i^2 \right)^\frac{1}{2}
\sup_{x_i \leq x} \left( e^{- s x_i} \right)
\\
& \qquad \qquad + 
\left(\sum_{x_i \geq x, i\leq N } \frac{d_i^2}{x_i^2}  
\right)^\frac{1}{2} 
\left(\sum_{x_i \geq x, i\leq N } f_i^2  \right)^\frac{1}{2}
\sup_{x_i \geq x, i\leq N  } \left(x_i e^{- s x_i} \right) \\ 
& \leq 
\|f\|_{\ell^2} N(x)^\frac{1}{2} \sup_{x_i \leq x} \left( e^{- s x_i} \right) + 
\|f\|_{\ell^2} M(x)^\frac{1}{2}  \sup_{x_i \geq x} \left(x_i e^{- s x_i} \right).
\end{align*} 
\end{proof}
We now prove injectivity. 
\begin{theorem}\label{inj} 
Let  $d_i> 0$,  let 
$\H$ be bounded, and assume that 
$x_i$ is monotonically rearrangeable. 
Then $\H$ is injective.
\end{theorem}
\begin{proof} 
 We 
 assume that the $x_i$ are already arranged in monotonically increasing order;
 at first consider the increasing case   
\begin{equation}\label{mono}
x_1 < x_2 \ldots  
\end{equation}
 We show that $\Nu(\H) = 0$.
We have 
with the Laplace transform $L$ defined in \eqref{lap}.
$\Nu(\H) = \Nu(L L^*) = \Nu(L^*)$.  
Let $f \in \Nu(\H)$, hence $f \in \Nu(L^*)$.

Assume that $f \not = 0$.
If we remove all indices with $f_i = 0$ (i.e., canceling rows and columns of $\H$)
from the sequences,  we can rewrite the problem using  corresponding 
subsequences of $x_i,d_i$ (keeping the same notation for them) 
and assuming now that $f_i \not = 0$ for all $i$. 
Note that \eqref{mono} holds also for such a subsequence.

By the assumption  that $f \in \Nu(L^*)$,  $f = (f_n)$, this sequences  satisfies 
\begin{equation}\label{abx} \sum_{i=1}^\infty f_i d_i e^{-x_i s} = 0 \qquad \text{for  a.e. } s >0,
\end{equation}
and the sum converges in the $L^2$-sense by continuity 
of $L^*$.  
By Lemma~\ref{lemmah1}, since $M(x)$ and $N(x)$ are finite, 
the series converges absolutely for $s >0$, and hence 
 the left-hand side in 
\eqref{abx} is a continuous function, and equality holds for all $s>0$. 

We may rewrite \eqref{mono} as 
\[ e^{-x_1 s} d_1  f_1 = - \sum_{i=2}^\infty f_i d_i e^{-x_i s}. \]
The right-hand side can be estimated by Lemma~\ref{lemmah1}, with
$M(x),N(x)$ defined as above but with $x_1$  removed; the corresponding 
terms are, however, upper-bounded by $M(x)$  and $N(x)$.  
Thus, we set $x = x_2$ to get  
\[ e^{-x_1 s} |d_1|  |f_1| \leq  
\|f\|_{\ell^2} N(x_2)^\frac{1}{2} e^{-s x_2} + 
\|f\|_{\ell^2} M(x_2)^\frac{1}{2}  \sup_{x_i  \geq x_2} x_i e^{-s x_i} \, .
\]
As the function $x \to x e^{-s x}$  is monotonically decreasing for $ x > s^{-1}$, 
with the setting  $s > x_2^{-1}$ and the monotonicity assumption for the $x_i$, we obtain 
using constants $C_1= \|f\|_{\ell^2} N(x_2)^\frac{1}{2},$ 
$C_2 = \|f\|_{\ell^2} M(x_2)^\frac{1}{2} $:
\begin{align*} |d_1|  |f_1| & \leq  
\left[ C_1  e^{-s x_2} + C x_2  e^{-s x_2} \right] e^{x_1 s} \\
& = 
\left[ C_1  e^{-s (x_2-x_1)} + C x_2  e^{-s (x_2-x_1)} \right]
\qquad \text{ for all } s > x_2^{-1}.
\end{align*}
Taking the limit $s \to \infty$, noting that $x_2 -x_1 >0$, 
yields  
 that $f_1  =0$, which is a contradiction to $f_i \not = 0$.
Thus, $f = 0$. 
In case that we can rearrange the $x_i$ into a increasing function, 
we make the transformation $\tilde{x_i} = x_i^{-1}$,
$\tilde{d_i} = \frac{d_i}{x_i}$ without changing $\H$. 
\end{proof}

Next, we discuss the ill-posedness of $\H$. 
\begin{theorem}
Let $d_i>0$, and let $\H$ be bounded and injective. 
If $x_i,d_i$ satisfy 
\begin{equation}\label{illcond} \liminf_{i} \frac{d_i}{\sqrt{x_i}} = 0, 
\end{equation}
then $\Ra(\H)$ is non-closed. 
\end{theorem} 
\begin{proof}
Suppose on the contrary that $\Ra(\H)$ is closed. 
Then $\Ra(\H) = \Ra(L L^*) = \Ra(L)$ and $\Ra(L)$ is closed as well. 
By the assumed injectivity of $\H$, 
we have that $\Ra(L)= \Ra(\H) = N(\H)^\bot = \ell^2$,
and by the closed range theorem, $L^{-1}$ is bounded. 
Consequently, for any sequence in $z \in \ell^2$, there is 
a $f_z \in L^2(0,\infty)$ with 
\[ z_k = d_k \int_0^\infty f_z(t) e^{- x_k t} dt ,  \quad \text{ and } 
\quad \|f_z\|_{L^2(0,\infty)} \leq C \|z\|_{\ell^2}. \]
Hence,
\[ |\frac{z_k}{d_k}| \leq \|f_z\||_{L^2(0,\infty)} \| e^{-x_k .}\|_{L^2(0,\infty)}
\leq C \|z\| \frac{1}{\sqrt{2 x_k}}. \]
Taking $z = e_k$ the $k$-th unit vector in $\ell^2$ yields that for all $k$,
\[  \frac{\sqrt{2}}{ C} \leq \frac{d_k}{\sqrt{x_k}},   \]
contradicting the assumption.
\end{proof} 


\begin{corollary}\label{colarry}
Let $d_i>0$ and  $\H$ be bounded. 
Let $x_k$ be a (infinite)  monotonically increasing rearrangeable sequence,
 and let either  $d_i$  or $x_i$ be 
bounded. 
Then $\H$ is ill-posed in the sense of Nashed, i.e., $\Ra(\H)$ is non-closed. 
\end{corollary}
\begin{proof}
By Theorem~\ref{inj}, $\H$ is injective. 
Since $N(x_1) \neq  0$, and the boundedness condition holds, $M(x_1) < CN(x_1)^{-1} < \infty$. 
This means that $\sum_{i=1}^\infty \frac{d_i^2}{x_i^2} $ converges, in particular 
$\frac{d_i}{x_i}$ must tend to $0$. If $x_i$ is bounded, then 
$\frac{d_i}{\sqrt{x_i}} = \frac{d_i}{x_i} \sqrt{x_i} \leq C   \frac{d_i}{x_i}$ 
and thus  \eqref{illcond} holds and $\H$ has non-closed range. 
Assume now that the $d_i$ are  bounded 
 and  that the $x_i$ are unbounded.  Then, $\frac{d_i}{\sqrt{x_i}} \leq C \frac{1}{\sqrt{x_i}} \to 0$
 as $i\to \infty$.
 Again by \eqref{illcond}
  $\H$ has non-closed range. 
\end{proof}
  
By the transformation \eqref{trans} we may state according conditions for the 
monotonically decreasing case as $x_i$ being bounded from below or 
$\frac{d_i}{x_i}$ being bounded.


\section{Examples and Consequences}
We provide some illustration of the results by further examples. 

Let us consider a Hilbert matrix generated by $d_i = 1$, $x_i = i^\alpha -\frac{1}{2}$,
where $\alpha >0$:
\[ (\H_{\alpha})_{i,j} = \frac{1}{i^\alpha + j^\alpha-1}. \]
The case $\alpha = 1$ is that of the  standard Hilbert matrix. 
We have for $\alpha >\frac{1}{2}$, 
\[ N(x_k) =  k , \] 
\[   \frac{C_1}{k^{2 \alpha +1}} =
\int_{k}^\infty \frac{1}{x^{2\alpha}-\frac{1}{2}} 
\leq M(x_k) = \sum_{i\geq k} \frac{1}{i^{2\alpha}-\frac{1}{2}}                            
\leq  \int_{k+1}^\infty \frac{1}{x^{2\alpha}-\frac{1}{2}}   \leq \frac{C_2}{(k+1)^{2 \alpha +1}}.
 \]
Thus, 
\[ 
  k^{2 - 2\alpha} = 
 k \frac{1}{k^{2\alpha-1}} \leq 
N(x_k) M(x_k) \leq C k \frac{1}{(k+1)^{2\alpha-1}} \leq C (k+1)^{2 - 2\alpha}, \]

For $\alpha >1$ these bounds tends to $0$ and they  
are bounded for $\alpha = 1$. The case of $\alpha <1$   always leads to unbounded $\H_\alpha$.
Thus, 
\[ \H_{\alpha} = \begin{cases} \text{unbounded} & \alpha < 1, \\ 
                               \text{noncompact} & \alpha = 1,\\ 
                               \text{compact} & \alpha >1, \end{cases} \]
and, for $\alpha \geq 1$ 
it is always ill-posed
in the sense of Nashed  by Corollary~\ref{colarry} since the $d_i$ are bounded.

Related is the sampling of the Laplace transform: 
\begin{corollary}
Consider the sampled Laplace transform 
 $L: L^2(0,\infty) \to \ell^2$ of \eqref{lap} with sampling points $x_i$. 

If the $x_i$ are sampled at shifted integer points $x_i = i -s$, $i \in \N $, $0 < s < 1$, 
then $L$ is a noncompact operator, and the inversion is an ill-posed problem 
of type I. 

If the $x_i$ are sampled at the following 
points with increasing distance $x_i = i^\alpha -s$, $i \in \N$, $0< s < 1$, 
with $\alpha >1$, 
then $L$ is a compact operator, and the inversion is a type-II ill-posed problem.  
\end{corollary}

A further example concerns a sideways 1-D heat equation. 
Consider the heat equation on the interval $[0,\pi]$ with 
Dirichlet boundary condition and initial conditions  $u_0 \in L^2$: 
\[ u_t(x,t) = \Delta u(x,t), \quad u(t=0) = u_0 \quad u(t,0) = u(t,\pi) = 0. \]
The solution can be found by a separation of variables. 
\[ u(x,t) = \sum_{k=1}^\infty  a_n \sin(n x) e^{- n^2 t}, \]
with  $a_n \in \ell^2$ the Fourier-sine coefficients of $u_0$. 

Consider now the recovery of $u_0$ from temporal measurements at the midpoint $x =\frac{\pi}{2}$, 
i.e., the data are $u(t,\frac{\pi}{2})$, $t \in \R$, and we use the $L^2(0,\infty)$-norm for the 
data space.
Then, the forward operator is 
\[ u_0 \to \sum_{n=0}^\infty  a_{2k+1} (-1)^k  e^{- (2 k +1)^2 t}, \]
This corresponds to  the mapping \eqref{lapad} 
$L^* \left((a_{2k+1} (-1)^k)_k \right)$ and $x_k = (2k +1)^2$. 
The mapping of $u_0 \in L^2$ to $a_{2k+1} (-1)^k \in \ell^2$ is bounded with continuous 
pseudo-inverse; thus the ill-posedness is determined by $L^*$, which, by the previous 
results is a compact (ill-posed) operator. Thus, this inverse problem is a type-II ill-posed problem.

Consider the similar problem with anomalous fractional diffusion (fractional heat equation) \cite{Vaz}: 
\[ u_t(x,t) = - (-\Delta )^\frac{1}{2} u(x,t), \quad u(t=0) = u_0 \quad u(t,0) = u(t,\pi) = 0. \]
The solution  is now given by 
\[ u(x,t) = \sum_{n=1}^\infty  a_{2k+1} \sin(nx) e^{- (2k+1) t}, \]
and the forward operator  corresponds as above to a mapping  $L^*$ with $x_k = 2 k +1$. 
Thus, since the $x_k$ are increasing linearly, 
in this case we obtain a non-compact ill-posed problem, i.e., a type-I problem.

\subsection{Singular values of $\H$} 
The result of Theorem~\ref{mainth} can also be used to estimate the 
singular values of a generalized Hilbert matrix.  For the case 
$d_n = \frac{1}{n}$ $x_n = n -\frac{1}{2}$ 
with the matrix $\H_{i,j} = \frac{1}{i} \frac{1}{i + j-1} \frac{1}{j}$, such 
estimates have been worked out in \cite{KiHo24} (cf.~similar results in 
\cite{HofMat22,Gerth22a}). 
It has been shown  (by considering the Hausdorff operator) 
in \cite{KiHo24} that 
\begin{equation}\label{upa} 
C \exp(-2 i)  \leq \sigma_i^\frac{1}{2}(\H) \leq \frac{C_2}{i^{3/2}}.   \end{equation}
This means that the decay of the eigenvalues of $\H$ 
is between polynomial and exponential, which 
leaves quite a gap of possibilities. 

Estimating the decay sharply 
is related to the question if the composition of 
type-I operators ``destroys'' the decay of type-II operator and has been studied 
by Hofmann from different perspectives \cite{HofWolf05,HofWolf09,HofMat22}. 
In particular, the ``{\em Hofmann problem}'' for the Hilbert matrix, i.e.,  
the question whether the singular values $\sigma_i(\H)$ decay polynomially or exponentially, 
(i.e., which of the upper or lower bound is sharper), 
it is an extremely difficult problem and still open.

Note that in the finite-dimensional case, it has been shown that the eigenvalues 
always decay exponentially~\cite{Becker},
but this does not help in the infinite-dimensional case 
as the accurate numerical computation of 
the smallest eigenvalues of  high-dimensional Hilbert matrices is almost impossible.

Using the established tools, we can find similar bound as in \eqref{upa} for general 
Hilbert matrices. 
\begin{theorem} \label{finalth}
Let $\H$ be the Hilbert matrix in \eqref{defHil} with the stated assumptions on $x_i,d_i$. 
Assume that $\H$ is compact. Denote by $\sigma_\myK(\H)$, $\myK = 1,\ldots $, the singular values (=eigenvalues) of $\H$ in decreasing order. We have the following estimate:
\begin{align*}
&\gamma_\myK \,\left(\frac{\frac{1}{2}\sum_{i=1}^\myK\frac{d_i^2}{ x_i} }{\myK-1}\right)^{-(\myK-1)} 
\leq \sigma_\myK(\H)  \leq \Bt(\myK) 
\end{align*}   
with $\Bt(\myK)$ defined in \eqref{btdef} and 
where $\gamma_\myK$  is the determinant of the first $\myK\times \myK$ principal submatrix of $\H$ given by 
\begin{equation}\label{detc}
\gamma_\myK =  \Pi_{i=1}^\myK d_i^2 
\frac{\Pi_{\myK\geq i>j \geq 1 } (x_i- x_j)^2}{\Pi_{1 \leq i,j \leq \myK} (x_j + x_j)}  .
\end{equation}
\end{theorem}
\begin{proof}
Denote by $\sigma_k(L)$, $k = 1,\ldots,$  the singular values in decreasing order of the 
sampled Laplace transform \eqref{lap}. 
By the Courant inequalities,  it follows that $\sigma_\myK(L) \leq \|(I - P_{\myK-1}) L\|$, where 
$P_\myK$ is a projector onto a finite-dimensional subspace. In particular, with $P_\myK$ from \eqref{prj} we obtain 
\[ \sigma_k(\H)^\frac{1}{2} = \sigma_k(L) \leq \|(I - P_{\myK-1}) L\| = \|(I - P_{\myK-1}) L L^*(I - P_{\myK-1})\|^\frac{1}{2}, \]
and the estimate follows from \eqref{thisesti}. 

For the lower bound,  let  $\H^\myK$ be the $\myK\times \myK$ Hilbert matrix obtained by taking the 
first $\myK$ rows and columns. Denote by $(\lambda_{i;\myK})_{i=1}^\myK$ its eigenvalues
in decreasing order.  Thus,  using the geometric-arithmetic mean inequality, 
and denoting by $\tr$ the trace, 
\begin{align*}
 &\det(\H^\myK) = \Pi_{i=1}^\myK \lambda_{i;\myK} = 
 \lambda_{\myK;\myK}  \Pi_{i=1}^{\myK-1} \lambda_{i;\myK} \leq 
   \lambda_{\myK;\myK}   \left(\frac{\sum_{i=1}^{\myK-1} \lambda_{i;\myK}}{\myK-1}\right)^{\myK-1} \\
   & \leq 
    \lambda_{\myK;\myK}   \left(\frac{\sum_{i=1}^{\myK} \lambda_{i;\myK}}{\myK-1}\right)^{\myK-1} = 
     \lambda_{\myK;\myK}   \left(\frac{\tr(\H^\myK)}{\myK-1}\right)^{\myK-1} 
     =    \lambda_{\myK;\myK}  \left(\frac{\sum_{i=1}^\myK \frac{1}{2}\frac{d_i^2}{ x_i} }{\myK-1}\right)^{\myK-1}. 
\end{align*}
The Cauchy interlacing theorem \cite[Corollary III.1.5]{Bhat} 
gives $ \lambda_{\myK;\myK} \leq \lambda_\myK (= \sigma_\myK(\H))$. The determinant of $\H^\myK$ is given by 
\eqref{detc} since  that of 
the Cauchy-matrix $\frac{1}{x_i + x_j}$ is well-known, e.g., \cite{Bjorck}, and 
the $d_i$ involve just an additional matrix multiplication by a diagonal matrix. 
This proves the result. 
\end{proof}

For the case in \cite{KiHo24} mentioned above, $d_n = \frac{1}{n}$ $x_n = n -\frac{1}{2}$,
we have that the upper bound is of the order of $\Bt(\myK) \sim \frac{1}{\myK^3}$ yielding a
similar   upper bound as in \eqref{upa}.
The lower bound can be estimated by $\sum_{i=1}^\myK\frac{d_i^2}{ x_i} \leq C$, 
the asymptotic formula for the determinant of the usual Hilbert matrix
$\sim C \myK^{-1/4} (2 \pi)^\myK 4^{-\myK^2}$,
and 
Stirling's approximation for $\Pi_{i=1}^\myK {d_i} = \frac{1}{\myK!}$ yielding 
\[ \gamma_\myK \sim 
  (2 \pi \myK)^{-1} \myK^{-2\myK} e^{2 \myK} \myK^{-1/4} (2 \pi)^\myK 4^{-\myK^2}. 
\]
Thus, we get a lower bound  of order $C e^{-C \myK^2}$, which is worse than that in \eqref{upa};
however, our  lower bound in Theorem~\ref{finalth} is valid for general sequences of $x_i,d_i$.


\bibliographystyle{siam}
\bibliography{Hilbert}

\end{document}